\documentclass[11pt]{amsart}

\usepackage{amssymb, latexsym, amsfonts, amsmath,amsthm,nicefrac}

\usepackage[vmargin=2.5cm,hmargin=2.5cm]{geometry}
\usepackage{color}
\usepackage{enumerate}

\newtheorem{Thm}{Theorem}
\newtheorem{Def}{Definition}

\newtheorem{Con}{Conjecture}
\newtheorem{Lem}{Lemma}

\newtheorem{Cor}{Corollary}

\theoremstyle{remark}

\newtheorem{Example}{Example}

\usepackage{mathrsfs}
\usepackage{fancyhdr}
\usepackage{enumerate}
\usepackage{enumitem}
\usepackage{verbatim}
\usepackage{amscd}
\usepackage{graphicx}

\newcommand{\NN}{\mathbb N}
\newcommand{\IN}{\mathbb Z}
\newcommand{\RN}{\mathbb R}
\newcommand{\CN}{\mathbb C}
\newcommand{\QN}{\mathbb Q}

\newcommand{\id}{\mathbb I}

\newcommand{\abs}[1]{\left|{#1}\right|}
\newcommand{\sgn}{\mathrm{sgn}}

\def\contFrac{%
    \operatornamewithlimits{%
        \mathchoice{% * Display style
            \vcenter{\hbox{\Large $\mathcal{K}$}}%
        }{%           * Text style
            \vcenter{\hbox{\Large $\mathcal{K}$}}%
        }{%           * Script style
            \mathrm{\mathcal{K}}%
        }{%           * Script script style
            \mathrm{\mathcal{K}}%
        }
    }
}

\newcommand{\cfplus}{\begin{array}{c}\\+\end{array}}
\newcommand{\cfdots}{\begin{array}{c}\\ \cdots\end{array}}
\newcommand{\gaussbkt}[1]{\left\lfloor {#1} \right\rfloor}
\newcommand{\lcm}[1]{\mathrm{lcm}\left({#1}\right)}
\newcommand{\genbinom}[2]{\left[\begin{array}{c}{#1} \\ {#2}\end{array}\right]}

\swapnumbers

\title[\resizebox{6in}{!}{Convergence Properties of the Classical and Generalized Rogers-Ramanujan Continued Fraction}]{Convergence Properties of the Classical and Generalized Rogers-Ramanujan Continued Fraction}
\author{Emil-Alexandru Ciolan}
\address{Rheinische Friedrich-Wilhelms-Universit\"at Bonn, Regina-Pacis-Weg 3, 53113 Bonn, Germany}
\email{ciolan@uni-bonn.de}

\author{Robert Axel Neiss}
\address{Mathematical Institute, University of Cologne, Weyertal 86-90, 50931 Cologne, Germany}
\email{rneiss@math.uni-koeln.de}

\begin{document}
\date{}

\begin{abstract}
The aim of this paper is to study the convergence and divergence of the Rogers-Ramanujan and the generalized Rogers-Ramanujan continued fractions on the unit circle. We provide an example of an uncountable set of measure zero on which the Rogers-Ramanujan continued fraction $R(x)$ diverges and which enlarges a set previously found by Bowman and Mc Laughlin. We further study the generalized Rogers-Ramanujan continued fractions $R_a(x)$ for roots of unity $a$ and give explicit convergence and divergence conditions. As such, we extend some work of Huang towards a question originally investigated by Ramanujan and some work of Schur on the convergence of $R(x)$ at roots of unity. In the end, we state several conjectures and possible directions for generalizing Schur's result to all Rogers-Ramanujan continued fractions $R_a(x)$.\end{abstract}
\maketitle

\section{Introduction}
%\noindent A finite continued fraction is an expression of the form  
%$$b_0+\dfrac{a_1}{b_1+\dfrac{a_2}{b_2+\dfrac{a_3}{b_3+\cdots+\dfrac{a_N}{b_N}}}},$$
%\subsection{Continued fractions.}
An infinite continued fraction is an expression of the form
$$b_0+\cfrac{a_1}{b_1+\cfrac{a_2}{b_2+\cfrac{a_3}{b_3+\cdots}}},$$
where $a_i$ and $b_i$ can be real or complex numbers, or functions of one or several variables, as will be the case throughout this paper.
For space considerations the following notation is used: 
\[ b_0+\dfrac{a_1}{b_1} \cfplus \dfrac{a_2}{b_2} \cfplus \cfdots.\]
We further adopt the notations
$$\contFrac_{i=1}^n \frac {a_i}{b_i} :=\dfrac{a_1}{b_1} \cfplus \dfrac{a_2}{b_2} \cfplus \cfdots \cfplus \dfrac{a_n}{b_n}\quad\text{and}\quad\contFrac_{i=1}^{\infty} \frac {a_i}{b_i} :=\dfrac{a_1}{b_1} \cfplus \dfrac{a_2}{b_2} \cfplus \cfdots . $$ 
As usual, we let $P_n$ and $Q_n$ be the unreduced numerator and denominator of the $n$-th \textit{convergent} (or \textit{approximant}) of the continued fraction, that is,
\begin{equation}\label{PnQn}
\frac{P_n}{Q_n}=b_0+\contFrac_{i=1}^{n}\frac{a_i}{b_i}. \end{equation}
\indent It is well-known (see, e.g., \cite[p. 9]{Lorentz}) that, for $n\ge 2,$ $P_n$ and $Q_n$ satisfy the following recursions
\begin{equation}
\begin{gathered}
P_n=b_nP_{n-1}+a_nP_{n-2},\\
Q_n=b_nQ_{n-1}+a_nQ_{n-2}
\end{gathered}
\end{equation}
and also that, for $n\ge 1,$
\begin{equation} \label{eqn:det-form}
P_nQ_{n-1}-P_{n-1}Q_n=(-1)^{n-1}\prod\limits_{i=1}^{n}a_i.
\end{equation}

\subsection{The Rogers-Ramanujan continued fraction and its generalization. }One of the most famous examples of continued fractions is the \textit{Rogers-Ramanujan continued fraction}, which is defined for $|x|<1$ by 
\begin{equation}\label{RogRaj}
R(x):=\frac{x^{1/5}}{1}\cfplus\frac{x}{1}\cfplus\frac{x^2}{1}\cfplus\frac{x^3}{1}\cfplus\cfdots.
\end{equation}    
The Rogers-Ramanujan continued fraction is known for its connections with the celebrated \textit{Rogers-Ramanujan identities}

\[G(q):=\sum_{n=0}^{\infty}\frac{q^{n^2}}{(q;q)_n}=\frac{
1}{(q;q^5)_{\infty}(q^4;q^5)_{\infty}},\]
\[H(q):=\sum_{n=0}^{\infty}\frac{q^{n^2+n}}{(q;q)_n}=\frac{1}{(q^2;q^5)_{\infty}(q^3;q^5)_{\infty}},\]
where 
\[(a;q)_0:=1,\quad (a;q)_n:=\prod_{j=0}^{n-1}(1-aq^j)\quad\text {and}\quad (a;q)_{\infty}:=\prod_{j=0}^{\infty}(1-aq^j),\quad |q|<1.\]
More precisely, if we let $K(x):=x^{1/5}/ R(x),$ we have 
\[K(q)=\frac{G(q)}{H(q)},\]
whence
\[R(q)=q^{1/5}\frac{(q;q^5)_{\infty}(q^4;q^5)_{\infty}}{(q^2;q^5)_{\infty}(q^3;q^5)_{\infty}}.\] 
The Rogers-Ramanujan continued fraction admits the following generalizations:
\[R_a(x):=\frac{1}{1}\cfplus\frac{ax}{1}\cfplus\frac{ax^2}{1}\cfplus\frac{ax^3}{1}\cfplus\cfdots\quad\text{(as defined in \cite[p. 12]{Berndt})},\] and
\[R(a):=\frac{a}{1}\cfplus\frac{ax}{1}\cfplus\frac{ax^2}{1}\cfplus\cfdots\quad\text{(as defined in \cite[p. 49]{Huang})}.\] 
For our purposes we shall use the former, namely 
\begin{equation}\label{genRR}
R_a(x)=\frac 11\cfplus \contFrac_{n=1}^{\infty}\frac{ax^n}{1}.\end{equation}
We introduce 
\[K_a(x):=\frac 11\cfplus\contFrac_{n=0}^{\infty}\frac{ax^n}{1}=\frac 11\cfplus aR_a(x). \]
%\subsection{Motivation.}
\indent For a survey on the Rogers-Ramanujan continued fraction and Rogers-Ramanujan identities, we refer the reader to \cite[Section 1]{Berndt} and \cite[pp. 3325--3327]{Bowman}.
We are interested in studying the convergence of the Rogers-Ramanujan continued fraction as well as its generalization. Let us recall that the continued fraction $\contFrac_{n=1}^{\infty}\frac{a_n}{b_n}$ is said to converge if $\frac{P_n}{Q_n}$ converges as $n\rightarrow\infty.$ It is clear that the convergence behavior of $R(x)$ (and $R_a(x)$) is the same with that of $K(x)$ (and $K_a(x)$). It is immediate by the following classical theorem that $R(x)$ (hence also $K(x)$) converges to a value in $\mathbb {\widehat{C}}$ for all $|x|<1$.  
\begin{Thm}[Worpitzky, {\cite[p. 35]{Lorentz}}]
Let the continued fraction $\contFrac_{n=1}^{\infty}a_n/1$ be such that $|a_n|\le 1/4$ for $n\ge 1$. Then $\contFrac_{n=1}^{\infty}a_n/1 $ converges. All approximants of the continued fraction lie in the disk $|w|<1/2,$ and the value of the continued fraction is in the disk $w\le 1/2$.
\end{Thm}
It is natural, however, to consider the expressions \eqref{RogRaj} and \eqref{genRR} also for $|x|\ge 1$ and ask questions about convergence. Andrews et al. \cite{Andrews} showed the following for $|x|>1$, thereby establishing a claim of Ramanujan:
\[\lim\limits_{j\rightarrow\infty}K_{2j+1}(x)=\frac{1}{K(-1/x)}\]
and
\[\lim\limits_{j\rightarrow\infty}K_{2j}(x)=\frac{K(1/x^4)}{x},\]
where 
\[K_n(x):=1+\frac{x}{1}\cfplus\frac{x^2}{1}\cfplus\frac{x^3}{1}\cfplus\cfdots\cfplus\frac{x^n}{1}.\]
\indent Thus, the interesting case to study convergence properties is when $\abs x = 1$, which is a much more subtle question.
%Therefore it is of interest to study convergence only for $|x|=1.$
Some work has been carried out in this regard. In an important 1917 paper of Schur \cite{Schur}, he showed that if $x$ is a primitive $m$-th root of unity and $m\equiv 0\pmod 5,$ then $K(x)$ diverges, and if $m\not\equiv 0\pmod 5,$ then $K(x)$ converges with 
\[K(x)=\lambda x^{(1-\lambda\sigma m)/5}K(\lambda),\]
where $\lambda=\left( \dfrac m5\right) $ is the Legendre symbol and $\sigma$ the least positive residue of $m$ modulo 5. It has been an open question since Schur's paper whether the Rogers-Ramanujan continued fraction converges or diverges at points on the unit circle which are not roots of unity. \\
\indent This convergence question was utilized by Lubinsky to provide an important counterexample to a conjecture of Baker-Gammel-Wills in \cite{Lubinsky}, and Bowman and Mc Laughlin performed a comprehensive study of convergence at roots of unity in \cite{Bowman}. In particular, they showed the following. 

\begin{Thm}[adapted from {\cite[Theorem 2]{Bowman}}]
\label{Bowman}
Let \[S=\left\lbrace t\in(0,1)\setminus\QN:e_{i+1}(t)\ge \phi^{d_i(t)}\text{ infinitely often}\right\rbrace, \]
where $\phi=\frac {1+\sqrt 5}2$ denotes the golden ratio,  $\left[0;e_1(t),e_2(t),\ldots\right] $ the continued fraction expansion of $t$ and $c_i(t)/d_i(t)$ its $i$-th convergent.   
Then $S$ is an uncountable set of measure zero, and if $t\in S$ and $y=\exp(2\pi i t)$ then $K(y)$ diverges.
\end{Thm}
 
In this paper we improve the above result by proving  
%Let $$\lambda_R:=\frac {1+R}2 + \sqrt{1+\left(\frac {1+R}2\right)^2},$$ $$T_R := \left\{t\in (0,1)\setminus \QN: \liminf_{n\rightarrow\infty} \frac {\lambda_R^{\frac {d_n(t)} 2}}{e_{n+1}(t)} < \infty \right\},$$ $$M_R: = \left\{x\in\CN: \abs x = 1, \abs{x+1} < R\right\}$$ and $$S_R: = \left\{e^{2\pi i t} : t\in T_R\right\} \cap M_R.$$  
%We will refer to these notations in Section 2, without defining them again. It will become then clear why we choose to define them like so. 

\begin{Thm}
\label{thm:rrf-divergence}
Let $R>0,$ $\lambda_R=\frac {1+R}2 + \sqrt{1+\left(\frac {1+R}2\right)^2}$ and 
\begin{equation*} %\label{eqn:div-set} 
S_R = \left\{e^{2\pi i t} : t\in T_R\right\} \cap M_R, \end{equation*}
where
\begin{equation*} M_R = \left\{x\in\CN: \abs x = 1, \abs{x+1} < R\right\} \text{~and~} T_R = \left\{t\in (0,1)\setminus \QN: \liminf_{n\rightarrow\infty} \frac {\lambda_R^{\frac {d_n(t)} 2}}{e_{n+1}(t)} < \infty \right\}. \end{equation*}
Then $S_R$ is an uncountable set of measure zero, and if $x\in S_R$ then $K(x)$ diverges.  
\end{Thm}
Theorem \ref{thm:rrf-divergence} improves Theorem 2, in the sense that the set $S_R$ extends the set $S$ locally around $-1$ for any $R\in(0,\sqrt 5-1)$. The smaller one chooses $R,$ the larger the set is around $-1$ on the unit circle. Note that $\lambda_R\to \phi$ as $R\to 0$.
\begin{Example}\setcounter{Example}{0}
We now give an example of $y\in S_R\setminus S.$ For this, choose
\begin{equation*} \begin{array}{lllll}
e_1 = 1, & e_2 = 1, & e_3 = 2, & e_{n+1} = \gaussbkt{\lambda_R^{d_n/2}} & \text{for }n\geq 3; \\
d_1 = 1, & d_2 = 2, & d_3 = 5, & d_{n+1} = e_{n+1}d_n+ d_{n-1} & \text{for }n\geq 3.
\end{array}\end{equation*}
Numerically, this is
\begin{equation*} t=[0;e_1,e_2,e_3,\dots] = [0;1,1,2,9,611180631,\dots]. \end{equation*}
This computation will be detailed in Section 2.\\
\end{Example} 
\noindent\textbf{Remark.} Note that the condition $t\in\mathbb R\setminus\mathbb Q$ is essential in both Theorems 2 and 3, because we need an infinite continued fraction expansion for $t$.\\

Despite the effort which has been expended on studying the convergence of the Rogers-Ramanujan continued fraction, very little is currently known about other continued fractions. As a generalization of what is known for the Rogers-Ramanujan continued fraction, a natural first step is to determine at which roots of unity does the generalized Rogers-Ramanujan continued fraction converge. We provide an extension of Schur's result, which in particular explains the 5-divisibility condition in his theorem as part of a larger framework. \\
%\noindent We also generalize Schur's result, by dealing with convergence of the generalized Rogers-Ramanujan continued fraction and proving the following
\indent The limit we compute in the following theorem coincides with what Ramanujan originally claimed in \cite[p. 57]{Ramanujan}. While stating the limit, he leaves open the question of where it actually exists. Huang partially answers this question in \cite{Huang} by using mostly analytic arguments. Here we come up with a purely algebraic approach which leads to different conditions; however, whenever Huang's and our conditions intersect, the results are consistent. 

\begin{Thm} \label{thm:gen-rrf-convergence}
Let $\zeta_m$ be a primitive $m$-th root of unity and $a\in\CN$. Assume $\sqrt{\frac 1 4 + a^m}\notin \QN(a,\zeta_m)$. Then 
$K_a$ converges at $x=\zeta_m$ to some limit in $\QN\left(a,\sqrt{\frac 1 4 + a^m},\zeta_m\right)\subset\CN 
$ if and only if $\frac 1 4 + a^m\notin \RN_{\leq 0}.$
In this case, the limit is given by
\begin{equation*} K_a(\zeta_m) = \frac {P_{m-2}(a,\zeta_m)} {\frac 1 2 + \sqrt{\frac 1 4 +a^m} - a\zeta_m^{m-1}P_{m-3}(a,\zeta_m)},\end{equation*}
where we set $P_{-1}(a,x):=1$ and $P_{-2}(a,x):=0$ for convenience.
\end{Thm} 
\begin{Thm}\setcounter{Thm}{2} \label{thm:gen-rrf-divergence}
Let $\zeta_m$ be a primitive $m$-th root of unity and $a\in\CN$. Assume $\frac 1 4 + a^m\in \RN_{<0}$. Then the generalized Rogers-Ramanujan continued fraction is divergent at $x=\zeta_m$. 
\end{Thm} 
Note now that Schur's result regarding the convergence of $K(x)$ in the case $m\not\equiv 0\pmod 5$ follows as a consequence of Theorem \ref{thm:gen-rrf-convergence} upon setting $a=1.$
\begin{Cor}\setcounter{Cor}{0} %\label{cor:schur-results}
Let $5\nmid m$ and $\zeta_m\in\CN$ be a primitive $m$-th root of unity. Then the Rogers-Ramanujan continued fraction is convergent at $x=\zeta_m$.
\end{Cor}

Briefly, the paper is organized as follows. In Section 2 we prove Theorem \ref{thm:rrf-divergence}. The proof requires several preliminary steps which are similar in spirit with the ideas in \cite[pp. 3331--3335]{Bowman}. In Section 3 we give the proofs of Theorems \ref{thm:gen-rrf-convergence} and \ref{thm:gen-rrf-divergence}. We conclude by giving a complete conjectural description of the convergence and divergence of $K_a(x)$ at roots of unity $a$. 

\section{Divergence of the Rogers-Ramanujan continued fraction}

In this section we study the divergence of the Rogers-Ramanujan continued fraction for $\abs x = 1$ and prove Theorem \ref{thm:rrf-divergence}. The following lemma will be a key step. 

\begin{Lem}
\label{lem:div-prop}
Let $\{a_n\}_{n\ge 1},\{b_n\}_{n\ge 1}\subset\CN,$ with $\abs{a_n}=1$. If the continued fraction
\begin{equation} \nonumber \frac 1 1 \cfplus \contFrac_{n=1}^\infty \frac{a_n}{b_n} \end{equation}
with unreduced numerator $P_N$ and denominator $Q_N$ converges, then
\begin{equation*} \label{eqn:denom-div} \lim_{N\rightarrow\infty} \abs{Q_N Q_{N-1}} = \infty. \end{equation*}
\begin{proof} This is essentially proved in \cite[p. 3331]{Bowman} but here we give a self-contained proof for the reader's convenience. 
By \eqref{eqn:det-form} we have
\begin{equation} \nonumber \abs{P_N Q_{N-1} - P_{N-1} Q_N} = \abs{(-1)^N \prod_{i=1}^N a_i} = 1. \end{equation}
Assuming that $P_N/Q_N\rightarrow L \in\CN,$ we find 
\begin{equation} \nonumber \abs{\frac 1 {Q_N Q_{N-1}}} = \abs{\frac {P_N Q_{N-1} - P_{N-1} Q_N} {Q_N Q_{N-1}}} = \abs{\frac {P_N}{Q_N} - \frac {P_{N-1}} {Q_{N-1}}} \leq \abs{\frac {P_N}{Q_N} - L} + \abs{\frac {P_{N-1}} {Q_{N-1}} - L} \stackrel{N\rightarrow\infty}\longrightarrow 0, \end{equation} 
yielding the desired claim.
\end{proof}
\end{Lem}

%%%%%%%%%%% rmk %%%%%%%%%%%%

Next, we recall the algorithm used to compute the continued fraction expansion of an irrational number in $(0,1).$ 
Let $t\in (0,1)\setminus\mathbb Q$. We define the recursive sequences
\begin{equation*} \label{eqn:real-contfrac} 
\begin{gathered}[l]
t_0(t) := t,\quad t_n(t) := \frac 1 {t_{n-1}(t)} - \gaussbkt{\frac 1  {t_{n-1}(t)}}\quad\text{for }n\ge 1,\\
e_n(t) := \gaussbkt{\frac 1 {t_{n-1}(t)}}\quad\text{for }n\ge 1. 
\end{gathered}
\end{equation*}
This implies
\begin{equation} \nonumber t = \contFrac_{n=1}^\infty \frac 1 {e_n(t)}. \end{equation}
If $c_n(t)$ and $d_n(t)$ denote the $n$-th numerator and denominator, then by \eqref{eqn:det-form} it is immediate that $\gcd(c_n(t),d_n(t))=1$.

%%%%%%%%%%%%%%% lem %%%%%%%%%%%%%%%%

\begin{Lem}
\label{lem:cf-exp-ineq}
With the previous notations, we have
\begin{equation*} \label{eqn:rcf-denom-ineq}  \abs{t-\frac{c_n(t)}{d_n(t)}} \leq \frac 1 {d_n(t)^2~e_{n+1}(t)}\quad\forall n\in\mathbb N. \end{equation*}
\begin{proof}
By the definitions and the recursion relation, we have
\begin{equation*} t = \contFrac_{k=1}^n \frac 1{e_k(t)} \cfplus \frac{t_n(t)}1 = \frac{c_{n}(t)+t_n(t)~c_{n-1}(t)} {d_n(t)+t_n(t)~d_{n-1}(t)}, \end{equation*} so that
\begin{align*} \abs{t-\frac {c_n(t)}{d_n(t)}} &= \abs{\frac{c_{n}(t)+t_n(t)~c_{n-1}(t)} {d_n(t)+t_n(t)~d_{n-1}(t)} - \frac {c_n(t)}{d_n(t)}} = \abs{\frac {t_n(t)~(c_{n-1}(t)~d_n(t)-c_n(t)~d_{n-1}(t))}{d_n(t)~(d_n(t)+t_n(t)~d_{n-1}(t))}} \\
		&\leq \frac {t_n(t)} {d_n(t)^2} \leq \frac 1 {d_n(t)^2~e_{n+1}(t)}. \qedhere \end{align*}
\end{proof}
\end{Lem}
Now, for $K(x)$ we have the recursion relations
\begin{equation*} \label{eqn:rr1-recursion}
\begin{array}{lll}
P_{n}(x) = P_{n-1}(x) + x^n P_{n-2}(x), & P_0(x) = 1, & P_{-1}(x) = 0; \\
Q_{n}(x) = Q_{n-1}(x) + x^n Q_{n-2}(x), & Q_0 (x)= 1, & Q_{-1}(x) = 1. \\
\end{array}
\end{equation*}
We introduce the notations $P_n(x)$ and $Q_n(x)$ for the convergents $P_n$ and $Q_n$ to point out the dependence on the variable $x$. Also note that defining the initial terms $P_{-1},P_0$ and $Q_{-1},Q_0$ does not affect the recursion.
We adapt an observation from Rogers in \cite[Lemma I]{Rogers} and recombine any three consecutive equations to obtain new recursion formulae
\begin{equation*} \label{eqn:rr2-recursion} 
\begin{gathered}
P_{n}(x) =  (1+x^{n-1}+x^n) P_{n-2}(x) - x^{2n-3} P_{n-4}(x);  \\
Q_{n}(x)=  (1+x^{n-1}+x^n) Q_{n-2}(x) - x^{2n-3} Q_{n-4}(x),  \\
\end{gathered}\end{equation*} with initial terms 
\begin{eqnarray}
\begin{array}{llll}\nonumber
P_2(x) = 1+x^2,   & P_1(x) = 1,    & P_0(x) = 1, & P_{-1}(x) = 0;  \\
Q_2(x) = 1+x+x^2, & Q_1(x) = 1+x,  & Q_0(x) = 1 ,& Q_{-1}(x) = 1.  \\
\end{array} 
\end{eqnarray}
As we decoupled the even and odd indexed convergents, it is now convenient to introduce a new notation:
$$ P_{n,k} := P_{2n+k},\quad Q_{n,k} := Q_{2n+k}\quad\text{for } n\in\NN\text{ and } k\in\{0,1\}. $$

%%%%%%%%%%%%% lem %%%%%%%%%%%%%

The following will provide a crucial bound for our computation.

\begin{Lem}
\label{lem:rec-growth}
Let $R>0$. Given the recursion
\begin{equation*} \alpha_n = (1+R)~\alpha_{n-1} + \alpha_{n-2} \text{ with } \alpha_1, \alpha_0 > 0, \end{equation*}
there exists a constant depending on $R$, say $C(R)>0,$ such that
\begin{equation*} 0 < \alpha_n < C(R)\lambda_R ^n\quad \forall n\in\NN.\end{equation*} 
%where    $\lambda_R  = \frac {1+R}2 + \sqrt{1+\left(\frac {1+R}2\right)^2} $ is as 
\end{Lem}
\begin{proof}
The characteristic equation of the sequence, $
\alpha^2-(1+R)\alpha-1=0,
$
has solutions $\lambda_R=\frac {1+R}2 + \sqrt{1+\left(\frac {1+R}2\right)^2}$ and $\lambda_R'=\frac {1+R}2 - \sqrt{1+\left(\frac {1+R}2\right)^2},$ therefore the general term will be equal to 
\[\alpha_n=a\cdot\lambda_R^n+b\cdot\lambda_R'^{n},\] for some constants $a,b$. 
\end{proof}
Lemma \ref{lem:rec-growth} applies to our problem in the following sense. Let $M_R=\left\{x\in\CN: \abs x = 1, \abs{x+1} < R\right\}$ be as in Section 1. 

\begin{Lem}
\label{lem:denom-bound}
For any $R>0$ there is a constant $C(R)>0$ depending on $R,$ such that
\begin{equation*}  \abs{Q_n(x)} \leq C(R)\lambda_R ^{\frac n 2}\quad \forall x\in M_R\text{ and } \forall n\in\NN. \end{equation*}
\begin{proof}
Apply Lemma \ref{lem:rec-growth} to the two sequences $\{Q_{n,0}(x)\}_{n\ge 0}$ and $\{Q_{n,1}(x)\}_{n\ge 0}$ separately. We only consider the case $\{Q_{n,0}(x)\}_{n\ge0}$ as the other sequence is handled similarly. 
For the sequence $\{\alpha_n\}_{n\ge0}$ in Lemma \ref{lem:rec-growth} choose the initial values
\begin{equation} \nonumber \alpha_0 = 1 = \abs{Q_{0,0}(x)},\quad \alpha_1 = 3 \geq \abs{Q_{1,0}(x)}.  \end{equation}
Then, by induction, one easily sees that, for $n\geq 2,$
\begin{align*} \abs{Q_{2n}(x)}  &= \abs{Q_{n,0}(x)} = \abs{(1 + x^{2n-1}(1+x))Q_{n-1,0}(x) - x^{4n-3} Q_{n-2,0}(x)} \\
				&\leq (1+\abs{x}^{2n-1}\abs{1+x}) \abs{Q_{n-1,0}(x)} + \abs{x}^{4n-3} \abs{Q_{n-2,0}(x)} \\
				&\leq (1+R)~\alpha_{n-1} + \alpha_{n-2} = \alpha_n \leq C(R) \lambda_R ^n. \qedhere
\end{align*}
\end{proof}
\end{Lem}

%%%%%%%%%%%%% lem %%%%%%%%%%%%%%

\begin{Lem}
\label{lem:denom-lipschitz}
For any $R>0$ there exist $A(R)>0$ and $\mu\geq 0$, such that
\begin{equation*} \label{eqn:denom-lipschitz}   \abs{Q_n(x)-Q_n(y)} \leq A(R)(n+\mu)^2\lambda_R ^{\frac n 2}\abs{x-y}\quad \forall x, y\in M_R\text{~and~}\forall n\in\NN. \end{equation*}
\begin{proof}
Let $x,y\in M_R$. To satisfy the claim for $n\in\{-1,0,1,2\}$, we only need to find the conditions
\begin{equation} \nonumber \mu \geq 0,~ A(R)\lambda_R  \geq 1+R \text{~~~and~~~} A(R)\lambda_R ^{\frac 1 2} \geq 1. \end{equation}
Thus, to verify the base case, we only need to bound $A(R)$ and $\mu$ from below. It is again necessary to prove the assertion for even and odd $n$ separately. Without loss of generality we deal only with the even case. Letting $n\geq 2$, we compute
\begin{align*} \abs{Q_{n,0}(x)-Q_{n,0}(y)} &\leq \abs{(1+x^{2n-1}(1+x))Q_{n-1,0}(x)-(1+y^{2n-1}(1+y))Q_{n-1,0}(y)} \\
					&+ \abs{x^{4n-3}Q_{n-2,0}(x)-y^{4n-3}Q_{n-2,0}(y)} \\
					&\leq \abs{Q_{n-1,0}(x)-Q_{n-1,0}(y)} + (2n-1)\abs{x-y}\abs{1+x}\abs{Q_{n-1,0}(x)} \\
					&+ \abs{y}^{2n-1} \abs{x-y} \abs{Q_{n-1,0}(x)} + \abs{y}^{2n-1} \abs{1+y} \abs{Q_{n-1,0}(x)-Q_{n-1,0}(y)} \\
					&+ (4n-3)\abs{x-y} \abs{Q_{n-2,0}(x)} + \abs{y}^{4n-3} \abs{Q_{n-2,0}(x)-Q_{n-2,0}(y)} \\
					&\leq (1+R) \abs{Q_{n-1,0}(x)-Q_{n-1,0}(y)} + (1+(2n-1)R)C(R)\lambda_R ^{n-1}\abs{x-y} \\
					&+ \abs{Q_{n-2,0}(x)-Q_{n-2,0}(y)} + (4n-3)C(R)\lambda_R ^{n-2}\abs{x-y},
\end{align*}
where the last inequality follows from Lemma \ref{lem:denom-bound}.\\
\indent In order to prove the assertion, we assume the induction hypothesis to hold for $n-1$ and $n-2$ and seek the sufficient conditions for $\mu$ and $A(R)$, for all $n$. We need the following inequality to hold:
\begin{align*} (1+R)~A(R)~(n+\mu-1)^2 \lambda_R ^{n-1} &+ (1+(2n-1)R)~C(R)~\lambda_R ^{n-1} + A(R)~(n+\mu-2)^2~\lambda_R ^{n-2} \\
&+ (4n-3)~C(R)~\lambda_R ^{n-2} \leq A(R)~(n+\mu)^2~\lambda_R ^n.
\end{align*}
Luckily, the terms of order $(n+\mu)^2$ cancel on both sides and by equivalent elementary manipulations it reduces to
\begin{align*} C(R)~\left[\left(4+2R\lambda_R \right)~n + \left((1-R)\lambda_R -3\right)\right]  
\leq A(R)~\left[\left(2(1+R)\lambda_R +4\right)(n+\mu) - \left(\lambda_R (1+R)+4\right)\right]. \end{align*}
But since we still have the freedom to choose $A(R)$ and $\mu$ very large, this condition can be satisfied for all $n$. Since the same calculation is possible for the odd sequence, the statement holds. 
\end{proof}
\end{Lem}

%%%%%%%%%%%%%% thm %%%%%%%%%%%%%%%

We finally prove the main divergence result. 

\begin{Thm}
Let $R>0,$ $\lambda_R=\frac {1+R}2 + \sqrt{1+\left(\frac {1+R}2\right)^2}$ and 
\begin{equation*} \label{eqn:div-set} S_R = \left\{e^{2\pi i t} : t\in T_R\right\} \cap M_R, \end{equation*}
where
\begin{equation*} M_R = \left\{x\in\CN: \abs x = 1, \abs{x+1} < R\right\} \text{~and~} T_R = \left\{t\in (0,1)\backslash \QN: \liminf_{n\rightarrow\infty} \frac {\lambda_R^{\frac {d_n(t)} 2}}{e_{n+1}(t)} < \infty \right\}. \end{equation*}
Then $S_R$ is an uncountable set of measure zero, and if $x\in S_R$ then $K(x)$ diverges.  \end{Thm}
Before giving the proof, we illustrate the example mentioned in the Introduction.
\begin{Example}
We construct an example with $x\in S_R\setminus S$. As for $s, t\in\RN,$ 
\begin{equation*} \abs{e^{2\pi is}-e^{2\pi it}} \leq 2\pi \abs{s-t}, \end{equation*}
we need $x=e^{2\pi it}$ with $t$ close to $\frac 1 2$. As $\frac{\sqrt 5 - 1}{2\pi}\approx 0.1967\dots$, we may pick $R=2\pi \frac {15}{100} < \sqrt 5 - 1$ and $t$ close to $\frac 3 5=[0;1,2,5]$. Thus we take
\begin{equation*} \begin{array}{lllll}
e_1 = 1, & e_2 = 1, & e_3 = 2, & e_{n+1} = \gaussbkt{\lambda_R^{d_n/2}} & \text{for }n\geq 3; \\
d_1 = 1, & d_2 = 2, & d_3 = 5, & d_{n+1} = e_{n+1}d_n+ d_{n-1} & \text{for }n\geq 3.
\end{array}\end{equation*}
Explicitly, this gives
\begin{equation*} t=[0;e_1,e_2,e_3,\dots] = [0;1,1,2,9,611180631,\dots] \end{equation*}
and
\begin{equation*} \abs{t-\frac 1 2} \leq \frac 1 {10} + \abs{t-\frac 3 5} = \frac 1 {10} + \abs{t-[0;1,1,2]} \leq \frac 1 {10} + \frac 1 {5\cdot47} = \frac {49}{470} < 0.15 = \frac R {2\pi}. \end{equation*}
Hence, $x=e^{2\pi it}\in S_R$. Because $\lambda_R^{\frac 1 2}<\phi=\frac{1+\sqrt 5}2$, we have
\begin{equation*} \liminf_{n\rightarrow\infty} \frac{\phi^{d_n}}{e_{n+1}} = \liminf_{n\rightarrow\infty} \frac{\lambda_R^{\frac{d_n}2}}{e_{n+1}}\frac{\phi^{d_n}}{\lambda_R^{\frac{d_n}{2}}} = \infty, \end{equation*}
therefore $x\notin S$.\end{Example}
Let us now conclude the section with the proof of the above Theorem. 
\begin{proof}[Proof of Theorem \ref{thm:rrf-divergence}]
Let $x=e^{2\pi it}\in S_R$. Again, $c_n(t), d_n(t)$ denote the unreduced numerators and denominators of the continued fraction expansion of $t\in(0,1)$. Define $x_n=e^{2\pi ic_n(t)/d_n(t)}$. Since $\gcd(c_n(t),d_n(t)) = 1$, $x_n$ is a primitive $d_n(t)$-th root of unity. Let $n\in\NN$. As a direct consequence of the explicit values computed by Schur in \cite[p. 134]{Schur} we find
\begin{equation} \label{eqn:schur} \max\left\{\abs{Q_{d_n(t)-1}(x_n)},\abs{Q_{d_n(t)-2}(x_n)}\right\} \leq 2. \end{equation}
Now, for $x_n\in M_R,$ in turn by Lemma \ref{lem:denom-lipschitz}, the fact that the chord length is shorter than the arc length and Lemma \ref{lem:cf-exp-ineq}, we have  
\begin{align} \abs{Q_{d_n(t)-1}(x_n)-Q_{d_n(t)-1}(x)} &\leq A(R)\left(d_n(t)-1+\mu\right)^2\lambda_R ^{\frac{d_n(t)-1}2}\abs{x_n-x} \nonumber \\
							&\leq 2\pi A(R) \left(d_n(t)-1+\mu\right)^2\lambda_R ^{\frac{d_n(t)-1}2} \abs{\frac{c_n(t)}{d_n(t)}-t} \nonumber \\
							&\leq 2\pi A(R)\left(1+\frac{\mu-1}{d_n(t)}\right)^2\lambda_R ^{\frac{d_n(t)-1}2}\frac 1 {e_{n+1}(t)} \label{eqn:proof-1}
\end{align}
and a similar result for $d_n(t)-2$. As such, since $t\in T_R$, for infinitely many $n,$ the right-hand side is bounded by a finite constant, say $B>0$. Then, for these $n$, using the triangle inequality and inequalities \eqref{eqn:schur} and \eqref{eqn:proof-1}, we see that
\begin{equation*} \abs{Q_{d_n(t)-1}(x)~Q_{d_n(t)-2}(x)} \leq (B+2)^2 \end{equation*}
and the proof concludes by Lemma \ref{lem:div-prop}. It is clear that $S_R$ is an uncountable set, because for any $n$ there are infinitely many options to successively choose $e_{n+1}(t)$ and different continued fraction expansions lead to different points in $S_R$. The fact that it has measure zero follows from \cite[Lemma 3]{Bowman}.
\end{proof}

%%%%%%%%%%%%%% Ex %%%%%%%%%%%%%%

%%%%%%%%%%%%%%%%%%%%%%%%%%%%%%%%%%%%%
%%%%%%%%%%%%%%%%%%%%%%%%%%%%%%%%%%%%%

\section{Convergence of the Generalized Rogers-Ramanujan Continued Fraction}

%%%%%%%%%%%%% defn %%%%%%%%%%

In this section we study the convergence of the generalized Rogers-Ramanujan fraction at roots of unity. More precisely, we study the convergence of 
\[K_a(x)=\frac 11\cfplus aR_a(x). \] 
\subsection{Preparations}
Before we start, we make a small adjustment of the approximants $P_n,Q_n$ of $K_a,$ as defined in \eqref{PnQn}, which will prove to be more convenient for our computations; namely, we let $$\frac{P_n(a,x)}{Q_n(a,x)}=\frac{1}{1}\cfplus\contFrac_{k=0}^{n}\frac{ax^k}{1}.$$ 
The notations $P_n(a,x)$ and $Q_n(a,x)$ are used to point out the dependence on the variables $a$ and $x$.\\
\indent We will need a special class of determinants closely related to the approximants of the numerator and denominator. 

\begin{Def}[{\cite[p. 120]{Schur}}] 
\label{defn:tridiag-det} 
Let $x_1,\dots,x_m$ be formal variables or complex numbers. Then define
\begin{equation} \label{eqn:tridiag-det}
D(x_1,\dots,x_m) := \det\left(\begin{array}{cccccc}
 1     & x_1    &   0    & \cdots & \cdots & 0 \\
-1     &   1    & x_2    & \ddots &        & \vdots \\
 0     &  -1    &   1    & \ddots & \ddots & \vdots \\
\vdots & \ddots & \ddots & \ddots & \ddots & 0 \\
\vdots &        & \ddots & \ddots &   1    & x_m \\
 0     & \cdots & \cdots &   0    &  -1    & 1
\end{array}\right).
\end{equation}
\end{Def}

%%%%%%%%% lem %%%%%%%%%

We want to summarize some of the properties of these determinants. The following can be found, for example in \cite[p. 11]{Perron}, but it is also cited in \cite[p. 133]{Schur}. For the sake of completeness, and as the original source is difficult to find, we include the proof here.

\begin{Lem}[{\cite[pp. 121, 133]{Schur}}]\label{lem:tridiag-det} 
Let $x_1,\dots,x_n$ be formal variables or complex numbers. Then:
\begin{equation*} \label{eqn:det-rec-rel} D(x_1,\dots,x_n) = D(x_1,\dots,x_{n-1}) + x_n~D(x_1,\dots,x_{n-2}) \end{equation*} and
\begin{equation} \label{eqn:det-slice-rel} D(x_1,\dots,x_n) = D(x_1,\dots,x_{m-1})~D(x_{m+1},\dots,x_n) + x_m~D(x_1,\dots,x_{m-2})~D(x_{m+2},\dots,x_n). \end{equation}
\begin{proof}
The first formula is obtained immediately by using the Laplace expansion on the last row or column. For the second formula we take a closer look into the combinatorics. 

The matrix in \eqref{eqn:tridiag-det} has tridiagonal form, i.e., only the entries on the main diagonal and the first diagonals above and below are non-zero. Then, in the definition of the determinant as a sum over permutations, $\sigma\in S_{n+1}$ contributes if and only if for all $i\in\{1,\dots,n+1\}$ we have $\abs{\sigma(i)-i} \leq 1$. This implies that $\sigma$ is a composition of disjoint transpositions of neighboring indices. We now distinguish $\sigma$ by what it does with the index $m$. Let $A=(a_{i,j})$ denote the matrix from (\ref{eqn:tridiag-det}). Then:
\begin{align*} 
D(x_1,\dots,x_n) &= \sum_{\sigma\in S_{n+1}} \sgn~\sigma\prod_{i=1}^{n+1} a_{i,\sigma(i)} = \sum_{\substack{\sigma\in S_{n+1} \\ \abs{\sigma(i)-i}\leq 1}} \sgn~\sigma\prod_{i=1}^{n+1} a_{i,\sigma(i)} \\
&= \sum_{\substack{\sigma\in S_{n+1} \\ \abs{\sigma(i)-i}\leq 1 \\ \sigma(m)\leq m}} \sgn~\sigma\prod_{i=1}^{n+1} a_{i,\sigma(i)} + \sum_{\substack{\sigma\in S_{n+1} \\ \abs{\sigma(i)-i}\leq 1 \\ \sigma(m)=m+1}} \sgn~\sigma\prod_{i=1}^{n+1} a_{i,\sigma(i)} \\
&= D(x_1,\dots,x_{m-1})D(x_{m+1},\dots,x_n) + x_mD(x_1,\dots,x_{m-2})D(x_{m+2},\dots,x_n).\qedhere
\end{align*}
\end{proof}
\end{Lem}

With the notations above we get
\begin{equation} \label{eqn:approx-det} P_n(a,x) = D(ax,\dots,ax^n) \text{~~~~and~~~~} Q_n(a,x) = D(a,ax,\dots,ax^n) \end{equation}
for the $n$-th approximants of the numerator and denominator. This is because the sequences on both sides of the equations satisfy the same recurrence relation and have the same initial values. These identities then imply a new recurrence relation.

%%%%%%%%%%%%%%%% lem %%%%%%%%%%%%

\begin{Lem} \label{lem:matr-rec}
Let $x$ be an $m$-th root of unity, $m\geq 3$. Then for $q\in\NN$ and $r\in\{0,\dots,m-1\}$, we have
\begin{equation} \label{eqn:matr-rec}
\left(\begin{array}{c} P_{(q+1)m+r}(a,x) \\ Q_{(q+1)m+r}(a,x) \end{array}\right) =
\left(\begin{array}{cc} ax^{m-1}P_{m-3}(a,x) & P_{m-2}(a,x) \\ ax^{m-1}Q_{m-3}(a,x) & Q_{m-2}(a,x) \end{array}\right) \left(\begin{array}{c} P_{qm+r}(a,x) \\ Q_{qm+r}(a,x) \end{array}\right).
\end{equation}
\begin{proof}
We use the representation \eqref{eqn:approx-det} of the approximants and use \eqref{eqn:det-slice-rel} to find that
\begin{align*} P_{(q+1)m+r}(a,x) &\stackrel{\text{\eqref{eqn:approx-det}}}= D(ax,\dots,ax^{(q+1)m+r}) \\
&\stackrel{\text{\eqref{eqn:det-slice-rel}}}= ax^{m-1}~D(ax,\dots,ax^{m-3})~D(ax^{m+1},\dots,ax^{(q+1)m+r}) \\
&+ D(ax,\dots,ax^{m-2})~D(ax^m,\dots,ax^{(q+1)m+r}) \\
&\stackrel{x^m=1}= ax^{m-1}~D(ax,\dots,ax^{m-3})~D(ax,\dots,ax^{qm+r}) \\
&+ D(ax,\dots,ax^{m-2})~D(a,\dots,ax^{qm+r}) \\
&= ax^{m-1}~P_{m-3}(a,x)~P_{qm+r}(a,x) + P_{m-2}(a,x)~Q_{qm+r}(a,x).
\end{align*}
The same procedure works for $Q_{(q+1)m+r}(a,x)$.
\end{proof}
\end{Lem}

%%%%%%%%%%%% lem %%%%%%%%%%%%%
In what follows, $\genbinom{m}k _x$ denotes the $x$-binomial coefficient (or generalized binomial coefficient), that is, 
\[\genbinom{m}k _x:=\frac{(x^m-1)\cdots(x^{m-k+1}-1)}{(x^k-1)\cdots(x-1)}.\] 
These results already occur in Ramanujan's original work \cite[p. 57]{Ramanujan} and are proved in \cite[Theorem 5.1]{Huang}.
\begin{Lem} [{\cite[Theorem 5.1]{Huang}}] \label{lem:qp-eqns}
Let $a,x$ be formal variables. Then $P_m(a,x), Q_m(a,x) \in\IN[x][a]$ are given by
\begin{equation*} \label{eqn:p-formula} P_m(a,x) = 1+\sum_{k=1}^{\gaussbkt{\frac {m+1} 2}} \genbinom{m+1-k}k _xx^{k^2}a^k \end{equation*} and
\begin{equation*} \label{eqn:q-formula} Q_m(a,x) = 1+\sum_{k=1}^{\gaussbkt{\frac m 2}+1} \genbinom{m+2-k}k _xx^{k(k-1)}a^k. \end{equation*}
\begin{proof}
One easily checks that the formulae hold for $m\in\{0,1\}$. Since both sequences satisfy the recursion relation
\begin{equation*} R_m(a,x) = R_{m-1}(a,x) + ax^m~R_{m-2}(a,x), \end{equation*}
the induction step is completed by using the well-known identity (see, for e.g., \cite[p. 128]{Schur})
\begin{equation}\label{binom} \genbinom m k _x = \genbinom{m-1}{k-1} _x~x^{m-k} + \genbinom{m-1}k _x. \qedhere \end{equation} \end{proof}
\end{Lem}

%%%%%%%%%%%%% lem %%%%%%%%%%%%%

It is essential to understand the eigenvalues of the matrix in \eqref{eqn:matr-rec}, as they will determine the limit behavior of the approximants $P_n$ and $Q_n$. The eigenvalues are determined by the characteristic polynomial. As we have a $2\times2$ matrix, we only need to know the trace and determinant. For $x$ any primitve $m$-th root of unity, the determinant is easily calculated from \eqref{eqn:det-form} to be
\begin{align} \nonumber \det\left(\begin{array}{cc} ax^{m-1}P_{m-3}(a,x) & P_{m-2}(a,x) \\ ax^{m-1}Q_{m-3}(a,x) & Q_{m-2}(a,x) \end{array}\right) &= -ax^{m-1}\left(P_{m-2}Q_{m-3}-P_{m-3}Q_{m-2}\right) \\\nonumber
&= (-1)^m a x^{m-1}\prod_{k=0}^{m-2}ax^k \\\nonumber
%\label{eqn:gen-rrf-det}  
&= (-1)^m a^m x^{\frac{m(m-1)}2} = -a^m. \end{align}
On invoking Lemma \ref{lem:qp-eqns} and recursion \eqref{binom} we can also determine the trace.

\begin{Lem} \label{thm:gen-rrf-tr}
Let $a, x$ be formal variables. Consider $P_n(a,x)$, $Q_n(a,x)$ as sequences in $\IN[x][a]$. For all $m\geq 3$ define
\begin{equation*} \label{eqn:gen-rrf-tr} T_m(a,x) = ax^{m-1}P_{m-3}(a,x)+Q_{m-2}(a,x) \in \IN[x][a]. \end{equation*}
Then $T_m(a,x)$ has the form
\begin{equation*} T_m(a,x) = 1 + \sum_{k=1}^{\gaussbkt{\frac m 2}} T_{m,k}(x)x^{k(k-1)}a^k, \end{equation*}
where 
\begin{equation*} T_{m,k}(x) = \frac {x^m-1}{x^{m-k}-1}~\genbinom{m-k}{k}_x. \end{equation*}
%Hence, all the $\{T_{m,1},\dots,T_{m,\gaussbkt{m/2}}\}$ are products of cyclotomic polynomials of order bounded by $m$ and the $m$-th cyclotomic polynomial is always a divisor.
\end{Lem}

%%%%%%%%%%%% rmk %%%%%%%%%%%%%%

%\noindent\textbf{Remark.} 
The fact that $T_m(a,x)=1$ for $x$ an $m$-th root of unity can be alternatively proven by invoking Theorem 2.2 in \cite{Huang}. Huang \cite{Huang} considers the following version of the generalized Rogers-Ramanujan continued fraction $$R(a)=\frac{a}{1}\cfplus \frac{ax}{1}\cfplus\frac{ax^2}{1}\cfplus\cfdots.$$ Let $p_n,q_n$ be the unreduced $n$-th numerator and denominator of $R(a)$, for $n\ge 1$. Note that \begin{equation} 
\label{connectionHuang}
K_a(x)=\frac{1}{1+R(a)}.
\end{equation} If we now let $\widetilde{P}_n,\widetilde{Q}_n$ be as in \eqref{PnQn}, we see that with our notation, for $n\ge 2,$ we have $\widetilde{P}_n=P_{n-2}$ and $\widetilde{Q}_n=Q_{n-2}.$ By \eqref{connectionHuang} it easily follows that 
\begin{equation*}
\label{newPQ}
\widetilde{P}_n=q_{n-1}\quad\text{and}\quad
\widetilde{Q}_n=p_{n-1}+q_{n-1}.
\end{equation*}
As such, 
\begin{align*}
T_m(a,x)&=ax^{m-1}P_{m-3}(a,x)+Q_{m-2}(a,x)=ax^{m-1}\widetilde{P}_{m-1}+\widetilde{Q}_{m}\\
&=ax^{m-1}q_{m-2}+(p_{m-1}+q_{m-1})=p_{m-1}+\left( q_{m-1}+ax^{m-1}q_{m-2}\right) \\
&=p_{m-1}+q_m=1.
\end{align*}  
For self-containment we recall some notations from \cite{Huang}. We define two sets
\[\mathcal{A}_n:=\left\lbrace \vec{v}=(n_1,\ldots,n_r)\in\mathbb{N}^r:r\ge 1,n_1=1,n_{i+1}-n_i\ge 2,\text{ and }n_r\le n\right\rbrace \]
and 
\[\mathcal{B}_n:=\left\lbrace \vec{v}=(n_1,\ldots,n_r)\in\mathbb{N}^r:r\ge 1,n_1\ge2,n_{i+1}-n_i\ge 2,\text{ and }n_r\le n\right\rbrace. \]
Next, let $\mathcal{A}_n(l)$ and $\mathcal{B}_n(l)$ be the subsets of $\mathcal{A}_n$ and $\mathcal{B}_n$ which contain all the $l$-dimensional vectors and $\mathcal{C}_n(l):=\mathcal{A}_{n-1}(l)\cup \mathcal{B}_n(l). $ By the proof of Theorem 2.2 in \cite{Huang}, 
we show that $$1 + \sum_{k=1}^{\gaussbkt{\frac m 2}} T_{m,k}(x)x^{k(k-1)}a^k=T_m(a,x)=1+\sum_{k=1}^{\gaussbkt{\frac m 2}}a^kx^{-k}\left( \sum_{\vec{v}\in\mathcal{C}_m(r)} x^{n_1+\cdots+n_r} \right), $$
from where the identity
\[\sum_{k=1}^{\gaussbkt{\frac m 2}} T_{m,k}(x)~x^{k(k-1)}~a^k=\sum_{k=1}^{\gaussbkt{\frac m 2}}a^kx^{-k}\left( \sum_{\vec{v}\in\mathcal{C}_m(r)} x^{n_1+\cdots+n_r} \right)\]
follows.\\

However, this should not come as a surprise, as the sets $\mathcal{A}_n(l)$ and $\mathcal{B}_n(l)$ parametrize precisely the permutations that contribute to the determinant formula \eqref{eqn:approx-det}, namely, the $n_i$ is the 
lower index of any transposition of neighboring indices, as explained 
in the proof of Lemma \ref{lem:tridiag-det}.

\subsection{Proof of the main convergence results}

%%%%%%%%%%%% thm %%%%%%%%%%%%%%

We are now ready to prove the following.

\begin{Thm} 
%\label{thm:gen-rrf-convergence}
Let $\zeta_m$ be a primitive $m$-th root of unity and $a\in\CN$. Assume $\sqrt{\frac 1 4 + a^m}\notin \QN(a,\zeta_m)$. Then 
$K_a$ converges at $x=\zeta_m$ to some limit in $\QN\left(a,\sqrt{\frac 1 4 + a^m},\zeta_m\right)\subset\CN 
$ if and only if $\frac 1 4 + a^m\notin \RN_{\leq 0}.$
In this case, the limit is given by
\begin{equation}\label{eqn:gen-rrf-limit} K_a(\zeta_m) = \frac {P_{m-2}(a,\zeta_m)} {\frac 1 2 + \sqrt{\frac 1 4 +a^m} - a\zeta_m^{m-1}P_{m-3}(a,\zeta_m)},\end{equation}
where we set $P_{-1}(a,x):=1$ and $P_{-2}(a,x):=0$ for convenience.
\end{Thm} 
\begin{proof}
Let $x=\zeta_m$ be a primitive $m$-th root of unity. If $m\geq 3$, the recursion from Lemma \ref{lem:matr-rec} holds. In this case we want to analyze the asymptotic behavior of $\{P_{qm+r}\}_q$ and $\{Q_{qm+r}\}_q$ using the matrix
\begin{equation*} A_m = \left(\begin{array}{cc} a\zeta_m^{m-1}P_{m-3}(a,\zeta_m) & P_{m-2}(a,\zeta_m) \\ a\zeta_m^{m-1}Q_{m-3}(a,\zeta_m) & Q_{m-2}(a,\zeta_m) \end{array}\right). \end{equation*}
Using the previous results, we compute the characteristic polynomial to be
\begin{equation*} \lambda^2 - \lambda - a^m = \left(\lambda - \frac 1 2 - \sqrt{\frac 1 4 + a^m}\right)\left(\lambda - \frac 1 2 + \sqrt{\frac 1 4 + a^m}\right) .\end{equation*}
By the assumption on $a^m$, the eigenvalues are different and thus the matrix can be diagonalized. We calculate the eigenvalues and eigenvectors to be
\begin{equation*} \lambda_\pm = \frac 1 2 \pm \sqrt{\frac 1 4 + a^m} \text{~~~~and~~~~} v_\pm = \left(\begin{array}{c} v_{1\pm} \\ v_{2\pm} \end{array}\right) = \left(\begin{array}{c} P_{m-2}(a,\zeta_m) \\ \lambda_\pm - a\zeta_m^{m-1} P_{m-3}(a,\zeta_m) \end{array}\right). \end{equation*}
Using Lemma \ref{lem:qp-eqns}, we see that the matrix has entries in the field $\QN(a,\zeta_m)$. By assumption, $\sqrt{\frac 1 4 + a^m}$ is not in this field. Hence the main diagonal of $A_m-\lambda_\pm\id_2$ cannot vanish. As the rows and columns of this matrix are multiples of each other, none of the entries can vanish, making the eigenvectors above well-defined. We note that the eigenvectors are forced to be in $\QN\left(a,\zeta_m,\sqrt{\frac 1 4+a^m}\right)^2\setminus\QN(a,\zeta_m)^2$. 
Consider now the eigenvector expansion for $r\in\{0,\dots,m-1\}:$
\begin{equation*} \left(\begin{array}{c} P_r(a,\zeta_m) \\ Q_r(a,\zeta_m) \end{array}\right) = a_{r+} v_+ + a_{r-} v_- \stackrel{\eqref{eqn:matr-rec}}\Longrightarrow \left(\begin{array}{c} P_{qm+r}(a,\zeta_m) \\ Q_{qm+r}(a,\zeta_m) \end{array}\right) = a_{r+} \lambda_+^q v_+ + a_{r-} \lambda_-^q v_-.\end{equation*}
As the left-hand side is in $\QN(a,\zeta_m)^2$, none of the coefficients $a_{r\pm}$ can vanish. We can now finally analyze the convergence behavior.
If we use the principal branch to define the complex square root, then $\abs{\lambda_+}\geq\abs{\lambda_-}$. Equality is equivalent to $\frac 1 4 + a^m \in \RN_{\leq 0}$. We find for all $r\in\{0,\dots,m-1\}$ that
\begin{equation*} \frac {P_{qm+r}(a,\zeta_m)} {Q_{qm+r}(a,\zeta_m)} = \frac {v_{1+} + \frac{a_{r-}}{a_{r+}}~\left(\frac{\lambda_-}{\lambda_+}\right)^q~v_{1-}} {v_{2+} + \frac{a_{r-}}{a_{r+}}~\left(\frac{\lambda_-}{\lambda_+}\right)^q~v_{2-}} .\end{equation*}
If $\abs{\lambda_+}>\abs{\lambda_-}$, the expression is convergent for all $r$ and the limits are equally given by the assertion. If $\abs{\lambda_+}=\abs{\lambda_-}$, the ratio can be written as $e^{i\varphi}$, $\varphi\in(-\pi,0)$ as $\operatorname{Re}\lambda_\pm=\frac 1 2$ and $\lambda_\pm\notin\QN$. We then find
\begin{align} \nonumber \abs{\frac{P_{(q+1)m+r}(a,\zeta_m)}{Q_{(q+1)m+r}(a,\zeta_m)} - \frac{P_{qm+r}(a,\zeta_m)}{Q_{qm+r}(a,\zeta_m)}} &= \abs{\frac{\frac{a_{r-}}{a_{r+}}e^{i\varphi q}(1-e^{i\varphi})~(v_{1+}v_{2-}-v_{1-}v_{2+})}{\left(v_{2+} + \frac{a_{r-}}{a_{r+}}e^{i\varphi (q+1)}v_{2-}\right)\left(v_{2+} + \frac{a_{r-}}{a_{r+}}e^{i\varphi q}v_{2-}\right)}} \\
\label{eqn:gen-rrf-div} &\geq \frac {\abs{a_{r-}}\abs{a_{r+}} \abs{1-e^{i\varphi}} \abs{\det(v_+|v_-)}} {\left(\abs{a_{r+}}\abs{v_{2+}} + \abs{a_{r-}}\abs{v_{2-}}\right)^2} > 0 .\end{align}
Thus the approximants are not Cauchy and in particular not convergent. \\

If $m=1$, \eqref{eqn:matr-rec} does not hold. Nevertheless, one obtains the same results by analyzing directly the recursion relations 
\begin{equation*} R_m(a,1) = R_{m-1}(a,1) + a R_{m-2}(a,1) \end{equation*} 
for $P$ and $Q$. This allows to directly calculate the limit for $\frac 1 4 + a\notin \RN_{\leq 0}$ by similar arguments as above. The divergence proof is similar. \\

If $m=2$, \eqref{eqn:matr-rec} is invalid again. Applying Rogers' trick -- the main idea of Section 2 -- leads to 
\begin{equation*} R_{2q+r}(a,-1) = R_{2(q-1)+r}(a,-1) + a^2 R_{2(q-2)+r}(a,-1) \end{equation*} 
for $P$ and $Q$. By elementary means, one directly computes convergence and the limit given above for $\sqrt{\frac 1 4 + a^2}\notin\QN(a)$. 
\end{proof}

%%%%%%%%%%%%%% thm %%%%%%%%%%%%%

As a partial converse to the previous result, we show the following. 

\begin{Thm} 
%\label{thm:gen-rrf-divergence}
Let $\zeta_m\in\CN$ be a primitive $m$-th root of unity and $a\in\CN$. Assume $\frac 1 4 + a^m\in \RN_{<0}$. Then the generalized Rogers-Ramanujan continued fraction is divergent at $x=\zeta_m$. 
\begin{proof}
The idea is to establish the inequality \eqref{eqn:gen-rrf-div} independent of the algebraic condition $\sqrt{\frac 1 4 + a^m}\notin\QN(a,\zeta_m).$ We need to find some $r\in\left\lbrace 0,1,\dots,m-1\right\rbrace $ such that both vectors in the eigenvector expansion do not vanish. The recursion relations for the $P_n$ and $Q_n$ can be written in vector form:
\begin{equation*} \left(\begin{array}{c} P_n(a,\zeta_m) \\ Q_n(a,\zeta_m) \end{array}\right) = \left(\begin{array}{c} P_{n-1}(a,\zeta_m) \\ Q_{n-1}(a,\zeta_m) \end{array}\right) + a\zeta_m^n\left(\begin{array}{c} P_{n-2}(a,\zeta_m) \\ Q_{n-2}(a,\zeta_m) \end{array}\right). \end{equation*}
Assume, for some $n,$ that $(P_n(a,\zeta_m), Q_n(a,\zeta_m))$ is an eigenvector of the matrix in \eqref{eqn:matr-rec}. The vector of index $n-1$ cannot be an eigenvector of the same eigenvalue, because the two vectors are linearly independent by the following determinant formula:
\begin{equation*} \det\left(\begin{array}{cc} P_n(a,\zeta_m) & P_{n-1}(a,\zeta_m) \\ Q_n(a,\zeta_m) & Q_{n-1}(a,\zeta_m) \end{array}\right) \stackrel{\eqref{eqn:det-form}}= (-1)^{n+1} a^{n+1} \zeta_m^{\frac{n(n+1)}2} \neq 0. \end{equation*}
If the index $n-1$ vector were an eigenvector of the other eigenvalue, then, in order to satisfy the recursion relation, the vector of index $n-2$ would need to have non-vanishing coefficients in front of both eigenvectors in the eigenvector expansion. Thus, we find an index $r\in\{0,\dots,m-1\}$ with
\begin{equation*} \left(\begin{array}{c} P_r(a,\zeta_m) \\ Q_r(a,\zeta_m) \end{array}\right) = a_{r+}v_+ + a_{r-}v_-, \text{~with~} a_{r\pm} \neq 0. \end{equation*}
Then \eqref{eqn:gen-rrf-div} can be similarly established for this $r$, proving the divergence.
\end{proof}
\end{Thm}

%%%%%%%%%%%% cor %%%%%%%%%%%%

%Theorem \ref{thm:gen-rrf-convergence} extends Schur's result for the Rogers-Ramanujan continued fraction, and places the 5-divisibility condition into a general structure. 

\begin{Cor} \label{cor:schur-results}
Let $5\nmid m$ and $\zeta_m\in\CN$ be a primitive $m$-th root of unity. Then the Rogers-Ramanujan continued fraction is convergent at $x=\zeta_m$.
\begin{proof}
We have
\begin{equation*} K_1(x) = \frac 1 {1 + x^{-\frac 1 5}R(x)}. \end{equation*}
Using the well-known fact that
\begin{equation*} \sqrt{\frac 1 4 + 1^m} = \frac 1 2 \sqrt 5 \in \QN(\zeta_m) \Leftrightarrow 5\mid m, \end{equation*}
the result immediately follows from Theorem \ref{thm:gen-rrf-convergence}. 
\end{proof}
\end{Cor}

%%%%%%%%%% rmk %%%%%%%%%%%%%%

\noindent{\bf Remark.} \label{rmk:algebraic-condition}
The containment condition on $\sqrt{\frac 1 4 + a^m}$ is natural in certain cases. For instance, if $a=1$, it identifies exactly the roots of unity where the continued fraction is not convergent, as it coincides with an earlier Theorem by Schur, see \cite[p. 135]{Schur}. \\
However, there are also cases where it is fulfilled and the fraction is still convergent. For example, let 
\begin{equation*} b\in\QN\cap\left(\frac 1 2, \sqrt{\frac 1 4 + \frac 1 {4^m}}\right). \end{equation*}
Then there is a unique real algebraic number $a\in\left( 0,\frac 1 4\right) $ with $\sqrt{\frac 1 4 + a^m}=b\in\QN$. But by \cite[Theorem 4.4]{Huang} the fraction is still convergent. The exact places where this condition is necessary remain a mystery.

%%%%%%%%%%%%%%%%%%%%%%%%%%%%%%%%%
%%%%%%%%%%%%%%%%%%%%%%%%%%%%%%%%%

\subsection{Open Questions}

The first open question, mentioned above, is to simplify or replace the algebraic condition $\sqrt{\frac 1 4 + a^m}\notin\QN(a,\zeta_m)$. It does not seem entirely natural and could be too strong. Nevertheless, we further examined the case when $a$ is a root of unity, as this is probably the most interesting (or at least, the easiest) one. Although we could not come up with a proof, we can conjecture the exact behavior in this case. 
%At the center of attention are the following two conjectures.
 
\begin{Con} \label{con:field-extension}
Let $\zeta_k, \zeta_m$ be primitive roots of unity. Then
\begin{equation*} \sqrt{1+4\zeta_k} \in \QN\left(\zeta_m\right) \Leftrightarrow k=1\text{~~~and~~~}5\mid m \text{~~~or~~~} k=2 \text{~~~and~~~} 3\mid m. \end{equation*}
\end{Con}

Note that only the ``$\Rightarrow$" direction is open. The reverse direction ``$\Leftarrow$" is explicitly solved (see the proof of Corollary \ref{cor:schur-results} above). 
The conjecture implies that, for $a=\zeta_k$ and $x=\zeta_m$ a root of unity, then
\begin{equation*} 
\sqrt{\frac 1 4+a^m} \in \QN\left(a,\zeta_m\right) = \QN\left(\zeta_{\lcm{k,m}}\right)
\Leftrightarrow a^m=1\text{~~~and~~~} 5\mid \lcm{k,m} \text{~~~or~~~} a^m=-1 \text{~~~and~~~} 3 \mid \lcm{k,m}. \end{equation*}
Here the second condition is not interesting, as this case is already covered by Theorem \ref{thm:gen-rrf-divergence}. The first condition can be resolved to $5\mid m$ and $k\mid m$. We want to state a specific conjecture for this case that we checked numerically using SAGE for $1\leq k\leq 50$ and $3\leq m\leq 100$.

\begin{Con}
Let $a=\zeta_k^j=e^{2\pi i \frac j k}$ and $x=\zeta_m^l=e^{2\pi i\frac l m}$ be roots of unity, $\gcd(j,k)=\gcd(l,m)=1$. Assume that $5\mid m$ and $k\mid m$. This allows us to define
\begin{equation*}
\pi_{\frac l m, \frac j k}: \begin{array}{ccc}
\IN/m\IN         & \rightarrow & \IN/k\IN \\
\left[l\right]_m & \mapsto     & \left[j\right]_k
\end{array} \text{~~~and~~~} \iota: \begin{array}{ccc}
\IN/k\IN         & \rightarrow & \IN/m\IN \\
\left[1\right]_k & \mapsto     & \left[\frac m k\right]_m
\end{array}.
\end{equation*}
Consider the set
\begin{equation*} R := \left(\iota\circ\pi_{\frac l m, \frac j k}\right)\left([-1]_m\right) \in \IN/m\IN. \end{equation*}
Then for any $r\in \{s-t: s\in R, t\in\{1,2\}\}\cap\NN$ we find
\begin{equation*} 
\left(\begin{array}{cc} ax^{m-1}P_{m-3}(a,x) & P_{m-2}(a,x) \\ ax^{m-1}Q_{m-3}(a,x) & Q_{m-2}(a,x) \end{array}\right) \left(\begin{array}{c} P_r(a,x) \\ Q_r(a,x) \end{array}\right) = \lambda\left(\begin{array}{c} P_r(a,x) \\ Q_r(a,x) \end{array}\right)\text{~~~~~for~}\lambda = \frac 1 2 \pm \sqrt{\frac 1 4 + a^m}.
\end{equation*}
Therefore we always find two consecutive indices such that the vectors are eigenvectors. This implies that $K_a(x)$ diverges.
\end{Con}

What are the consequences of this conjecture? By the determinant formula \eqref{eqn:det-form} two consecutive vectors cannot be linearly dependent. Hence, the vectors must be eigenvectors of different eigenvalues. By the assumption that $5\mid m$ and $k\mid m$, the eigenvalues are real-valued and have different absolute values. Therefore, the consecutive vectors will produce two different limit points in the sequence $\{P_n/Q_n\}_n$, to be specific, the ratios of the first and second entry of the linearly independent eigenvectors. Therefore, the continued fraction is divergent. In conclusion, these two conjectures imply the following general picture, which predicts in an explicit manner the set of limit points of $K_a(x)$, for any $a$ and $x$ being roots of unity.

\begin{Con}
Let $a=\zeta_k, x=\zeta_m$ be primitive roots of unity, $m\geq 3$. Then
\begin{align*} K_a(x) \text{~converges} &\Leftrightarrow \sqrt{\frac 1 4 + a^m}\notin\QN(a,x) \text{~and~} \frac 1 4 + a^m \notin\RN_{\leq 0} \\ &\Leftrightarrow 5\nmid m \text{~or~} k\nmid m \text{~and~} a^m\neq -1. \end{align*}
In case of convergence, the limit is given by \eqref{eqn:gen-rrf-limit}. If $5\mid m$ and $k\mid m$, the two limit points of $\left\lbrace P_n/Q_n\right\rbrace _n$ are exactly
\begin{equation*} \frac {P_{m-2}(a,x)} {\frac 1 2 \pm \sqrt{\frac 1 4 +a^m} - ax^{m-1}P_{m-3}(a,x)} \in \widehat{\CN}.\end{equation*}
For $3\mid m$ and $a^m=-1$ these limit points can also occur. If $\{P_{qm+r}/Q_{qm+r}\}_q$ does not converge to one of these limit points, it produces exactly these three more limit points
\begin{equation*} \frac {a_{r+} v_{1+} + a_{r-} v_{1-} \zeta_3^s}{a_{r+} v_{2+} + a_{r-} v_{2-} \zeta_3^s} \text{~for~} s\in\{0,1,2\}, \end{equation*}
where
\begin{equation*} v_\pm = \left(\begin{array}{c} v_{1\pm} \\ v_{2\pm} \end{array}\right) = \left(\begin{array}{c} P_{m-2}(a,x) \\ \frac 1 2 \pm \sqrt{\frac 1 4+a^m} - ax^{m-1}P_{m-3}(a,x) \end{array}\right) \text{~and~} \left(\begin{array}{c} P_r(a,x) \\ Q_r(a,x) \end{array}\right) = a_{r+} v_+ + a_{r-} v_-. \end{equation*}
\end{Con}
\textbf{Acknowledgements.} This paper was written during the Cologne Young Researchers in Number Theory Program 2015, organized by Larry Rolen. The authors would like to thank him for his tireless support and numerous ideas regarding the project and the University of Cologne for hospitality. The program was funded by the University of Cologne postdoc grant DFG Grant D-72133-G-403-151001011, under the Institutional Strategy of the University of Cologne within the German Excellence Initiative.\\
\indent The authors would also like to thank Kathrin Bringmann, Michael Griffin and Armin Straub for the fruitful discussions and helpful suggestions. In particular, Armin Straub suggested that the polynomials $T_{m,k}$ in Lemma \ref{thm:gen-rrf-tr} might factor into cyclotomic polynomials and Michael Griffin proposed the precise formula for $T_{m,k}$. The authors had further inspiring discussions with Karl Mahlburg, whom they would like to thank for his comments. Last, but not least, the authors thank the anonymous referees for the useful observations and suggested corrections.

\end{document}